\documentclass[english,sumlimits,reqno]{amsart}

\usepackage[T1]{fontenc}
\usepackage[utf8]{inputenc}
\usepackage{lmodern}

\usepackage{babel}

\usepackage[a4paper]{geometry}

\usepackage{xifthen}
\usepackage{ifthenx}

\usepackage{xargs}

\usepackage{graphicx}
\usepackage[table,svgnames,x11names]{xcolor}

\usepackage[disable,colorinlistoftodos,prependcaption,textsize=tiny]{todonotes}
\newcommandx{\todoin}[2][1=]{\todo[inline, caption={todo}, #1]{%
    \begin{minipage}{\textwidth-20pt}#2\end{minipage}}}

\newcommandx{\remove}[2][1=]{\todo[linecolor=Plum,backgroundcolor=Plum!25,bordercolor=Plum,#1]{#2}}
\newcommandx{\removein}[2][1=]{\remove[inline, caption={todo}, #1]{%
    \begin{minipage}{\textwidth-20pt}#2\end{minipage}}}

\usepackage{amssymb}
\usepackage{amsmath}
\usepackage{amsthm}
\usepackage{mathtools}
\usepackage{exscale}
\usepackage{relsize}
\usepackage{bbm}
\usepackage{bm}
\usepackage{amsbsy}
\usepackage{mathdots}

\usepackage{mathrsfs}

\usepackage{stmaryrd}

\usepackage{mdframed}

\usepackage[all]{xy}

\usepackage{fp}

\usepackage[section]{placeins}
\usepackage{float}

\usepackage{enumerate}

\newcounter{proof}
\newenvironment{myproof}%
{\stepcounter{proof}\begin{proof}}%
{\end{proof}}%
\newcounter{proofstep}[proof]
\newenvironment{proofstep}[1][]%
{\refstepcounter{proofstep}\bigskip\par\noindent%
  \ifthenelse{\isempty{#1}}
    {\textsc{Step \theproofstep. }}
    {\textsc{#1.}}
  \noindent}%
{\par}%
\newcounter{proofcase}[proof]
{\refstepcounter{proofcase}\bigskip\par\noindent%
  \ifthenelse{\isempty{#1}}
    {\textsc{Case \theproofcase. }}
    {\textsc{#1.}}
  \noindent}%
{\par}%

\usepackage{varioref}
\usepackage{hyperref}
\hypersetup{
  colorlinks,
  allcolors=DarkBlue
}
\usepackage[nameinlink]{cleveref}

\theoremstyle{plain}
\newtheorem{thm}{Theorem}[section]
\newtheorem*{thm*}{Theorem}
\newtheorem{pro}[thm]{Proposition}
\newtheorem{cor}[thm]{Corollary}
\newtheorem{lem}[thm]{Lemma}

\theoremstyle{definition}

\theoremstyle{remark}

\newtheorem{rem}[thm]{Remark}

\numberwithin{equation}{section}

\newcommandx{\textref}[2][1=]{\hyperref[#2]{#1\ref*{#2}}}
\newcommandx{\textrefp}[2][1=]{(\hyperref[#2]{#1\ref*{#2}})}

\DeclareMathOperator{\sign}{sign}

\DeclareMathOperator{\spn}{span}

\DeclareMathOperator{\cond}{\mathbb{E}}
\DeclareMathOperator{\prob}{\mathbb{P}}

\newcommand{\bmo}{\ensuremath{\mathrm{BMO}}}

\begin{document}

\title[Subsymmetric bases have the factorization property]{Subsymmetric bases have the factorization
  property}

\author[R.~Lechner]{Richard Lechner}

\address{Richard Lechner, Institute of Analysis, Johannes Kepler University Linz, Altenberger
  Strasse 69, A-4040 Linz, Austria}

\email{Richard.Lechner@jku.at}

\date{\today}



\thanks{Supported by the Austrian Science Foundation (FWF) Pr.Nr. P32728}

\begin{abstract}
  We show that every subsymmetric Schauder basis $(e_j)$ of a Banach space $X$ has the factorization
  property, i.e.\@ $I_X$ factors through every bounded operator $T\colon X\to X$ with a
  $\delta$-large diagonal (that is $\inf_j |\langle Te_j, e_j^*\rangle| \geq \delta > 0$, where the
  $(e_j^*)$ are the biorthogonal functionals to $(e_j)$).  Even if $X$ is a non-separable dual space
  with a subsymmetric weak$^*$ Schauder basis $(e_j)$, we prove that if $(e_j)$ is
  non-$\ell^1$-splicing (there is no disjointly supported $\ell^1$-sequence in $X$), then $(e_j)$
  has the factorization property.  The same is true for $\ell^p$-direct sums of such Banach spaces
  for all $1\leq p\leq \infty$.

  Moreover, we find a condition for an unconditional basis $(e_j)_{j=1}^n$ of a Banach space $X_n$
  in terms of the quantities $\|e_1+\ldots+e_n\|$ and $\|e_1^*+\ldots+e_n^*\|$ under which an
  operator $T\colon X_n\to X_n$ with $\delta$-large diagonal can be inverted when restricted to
  $X_\sigma = [e_j : j\in\sigma]$ for a ``large'' set $\sigma\subset \{1,\ldots,n\}$ (restricted
  invertibility of $T$; see Bourgain and Tzafriri~[{\em Israel J.\@ Math.}~1987, {\em London Math.\@
    Soc.\@ Lecture Note Ser.}~1989).  We then apply this result to subsymmetric bases to obtain that
  operators $T$ with a $\delta$-large diagonal defined on any space $X_n$ with a subsymmetric basis
  $(e_j)$ can be inverted on $X_\sigma$ for some $\sigma$ with $|\sigma|\geq c n^{1/4}$.
\end{abstract}

\maketitle


\makeatletter
\providecommand\@dotsep{5}
\def\listtodoname{List of Todos}
\def\listoftodos{\@starttoc{tdo}\listtodoname}
\makeatother

\newcounter{mycounter}%

\section{Introduction}
\label{sec:main-results}

Throughout this paper, we assume that $X$ and $Y$ are Banach spaces satisfying the following
properties \textrefp[B]{item:1}--\textrefp[B]{item:5}.  We assume that there exists a bilinear map
$\langle\cdot, \cdot\rangle\colon X\times Y\to \mathbb{R}$ such that:
\begin{enumerate}[(B1)]
\item\label{item:1} whenever $x\in X$ and $\langle x, y\rangle = 0$ for all $y\in Y$, then $x = 0$;
\item\label{item:2} whenever $y\in Y$ and $\langle x, y\rangle = 0$ for all $x\in X$, then $y = 0$;
\item\label{item:3} there exists a constant $C_d > 0$ such that
  $|\langle x, y\rangle|\leq C_d \|x\|_X\|y\|_Y$ for all $x\in X$, $y\in Y$.
  \setcounter{mycounter}{\value{enumi}}
\end{enumerate}
By $\sigma(X,Y)$, we denote the locally convex topology on $X$ generated by the collection of
seminorms $\{x\mapsto |\langle x, y \rangle| : y\in Y\}$.  In addition, we assume there exist
normalized sequences $(e_j)$ in $X$ and $(f_j)$ in $Y$ such that
\begin{enumerate}[(B1)]
  \setcounter{enumi}{\themycounter}
\item \label{item:4} $\langle e_j, f_j\rangle = 1$ and $\langle e_j, f_k\rangle = 0$, for all
  $j\neq k$;
\item \label{item:5} every $x\in X$ has the unique representation
  $x = \sum_{j=1}^\infty \langle x, f_j\rangle e_j$, where the series converges in the
  $\sigma(X,Y)$-topology.  \setcounter{mycounter}{\value{enumi}}
\end{enumerate}
If $(z_j)$ is a normalized Schauder basis for the Banach space $Z$, $Z^*$ is the dual space and
$(z_j^*)$ denotes the biorthogonal functionals of $(z_j)$ ($(z_j^*)$ is called a weak$^*$ Schauder
basis for $Z^*$), then \textrefp[B]{item:1}--\textrefp[B]{item:5} are satisfied for $X=Z$, $Y=Z^*$,
$e_j=z_j$, $f_j = z_j^*$, $j\in\mathbb{N}$ and $C_d = 1$.  The same is true for $X=Z^*$, $Y=Z$
$e_j=z_j^*$, $f_j = z_j$, $j\in\mathbb{N}$ and $C_d = 1$.  For more information on topological bases
we refer to~\cite[Section~1.b]{lindenstrauss:tzafriri:1977} and~\cite[Chapter~1,~§13]{singer:1970}.

For each set $\mathcal{A}$, $|\mathcal{A}|$ denotes the cardinality of that set.  Given a sequence
of vectors $(x_j)$ in a Banach space $X$ and $\mathcal{A}\subset\mathbb{N}$,
$[x_j : j\in \mathcal{A}]$ denotes the norm-closure of $\spn\{x_j : j\in\mathcal{A}\}$.

We say that the sequence $(e_j)$ is \emph{$C_u$-unconditional}, if for all sequences of scalars
$(a_j)$, $(\gamma_j)$ holds that
\begin{equation}\label{eq:unconditional}
  \Big\| \sum_{j=1}^\infty \gamma_j a_j e_j \Big\|_X
  \leq C_u \sup_{k} |\gamma_k| \Big\| \sum_{j=1}^\infty a_j e_j \Big\|_X,
\end{equation}
where we demand the above series converge in the $\sigma(X,Y)$-topology.  Moreover, we say that
$(e_j)$ is \emph{$C_s$-spreading}, if $(e_j)$ is $C_s$-equivalent to each of its increasing
subsequences, i.e.\@
\begin{equation}\label{eq:spreading}
  \frac{1}{C_s} \Bigl\|\sum_{j=1}^\infty a_j e_{n_j}\Bigr\|_X
  \leq \Bigl\|\sum_{j=1}^\infty a_j e_j\Bigr\|_X
  \leq C_s \Bigl\|\sum_{j=1}^\infty a_j e_{n_j}\Bigr\|_X,
\end{equation}
for all increasing $(n_j)$.  Again, we demand that the above series converge in the
$\sigma(X,Y)$-topology.  If~\eqref{eq:unconditional} and~\eqref{eq:spreading} are both satisfied, we
say that $(e_j)$ is \emph{$(C_u,C_s)$-subsymmetric}.  If $(e_j)$ is $C$-unconditional or
$C$-spreading or $(C,C)$-subsymmetric for some $C$, we say that $(e_j)$ is unconditional or
spreading or subsymmetric.  For more information on unconditional or subsymmetric bases we refer
to~\cite{lindenstrauss:tzafriri:1977,singer:1970}.

Moreover, we introduce the following notions.  Let $I_X$ denote the identity operator on the Banach
space $X$ and let $T\colon X\to X$ denote a bounded linear operator.  We say that
\begin{itemize}
\item $I_X$ \emph{$C$-factors through} $T$, if there exist operators $A,B\colon X\to X$ such that
  $I_X = ATB$ and $\|A\|\cdot\|B\|\leq C$.
\item $I_X$ \emph{factors through} $T$ if $I_X$ \emph{$C$-factors through} $T$ for some $C$.
\item $I_X$ \emph{almost $C$-factors through} $T$ if $I_X$ $(C+\eta)$-factors through $T$, whenever
  $\eta > 0$.
\item $T$ has \emph{$\delta$-large diagonal} (with respect to $(e_j)$) if
  $\delta:= \inf_j |\langle T e_j, f_j\rangle| > 0$.
\item $T$ has \emph{large diagonal} (with respect to $(e_j)$) if $T$ has $\delta$-large diagonal for
  some $\delta > 0$.
\item $(e_j)$ has the \emph{$K(\delta)$-factorization property} if $I_X$ almost $K(\delta)$-factors
  through every bounded linear operator $T:X\to X$ which has $\delta$-large diagonal with respect to
  $(e_j)$.
\item $(e_j)$ has the \emph{factorization property} if $I_X$ factors through every bounded linear
  operator $T:X\to X$ which has a large diagonal with respect to $(e_j)$.
\item $X$ is \emph{primary}, if for every bounded projection $Q\colon X\to X$ either $Q(X)$ or
  $(I_X - Q)(X)$ is isomorphic to $X$ (see~\cite[Definition~3.b.7]{lindenstrauss:tzafriri:1977}).
\end{itemize}
In this context, the notion of a large diagonal first appeared implicitly in~\cite{andrew:1979} and
was then later formally introduced in~\cite{laustsen:lechner:mueller:2015}.  The factorization
property was conceived in~\cite{lechner:motakis:mueller:schlumprecht:2020} and further investigated
in~\cite{lechner:motakis:mueller:schlumprecht:2019}.

Finally, we will define non-$\ell^1$-splicing bases, which were introduced and investigated
in~\cite{lechner:2019:subsymm} (specifically for weak$^*$ Schauder bases).  We would like to mention
that prior to formally introducing non-$\ell^1$-splicing, the underlying properties of such bases
were exploited e.g.\@ in~\cite{lechner:2018:SL-infty, bourgain:1983} and can be traced back all the
way to Lindenstrauss' proof showing that $\ell^\infty$ is prime~\cite{lindenstrauss:1967} (see
also~\cite[Theorem~2.a.7]{lindenstrauss:tzafriri:1977}).

We assume now that $(e_j)$ is unconditional and define for each $\Lambda\subset\mathbb{N}$ the
bounded projection $P_\Lambda\colon X\to X$ by
\begin{equation}\label{eq:definition-P_A}
  P_\Lambda\Big( \sum_{j=1}^\infty a_j e_j \Big)
  = \sum_{j\in\Lambda} a_j e_j,
\end{equation}
where the above series converge in the $\sigma(X,Y)$-topology.  We then say that the unconditional
sequence $(e_j)$ in $X$ is \emph{non-$\ell^1$-splicing} if for every infinite set
$\mathcal{I}\subset\mathbb{N}$ and every $\theta > 0$ we can find a sequence $(\Lambda_j)$ of
pairwise disjoint and infinite subsets of $\mathcal{I}$, such that for all sequences
$(x_j)\subset X$ with $\|x_j\|_{X}\leq 1$, $j\in\mathbb{N}$ there exists a sequence of scalars
$(a_j)\in \ell^1$ with $\|(a_j)\|_{\ell^1} = 1$ such that
\begin{equation}\label{eq:condition-c}
  \Big\| \sum_{j=1}^\infty a_j P_{\Lambda_j} x_j \Big\|_{X}
  \leq \theta.
\end{equation}
Otherwise, we say that $(e_j)_{j=1}^\infty$ is \emph{$\ell^1$-splicing}.

If $(e_j)$ is subsymmetric, J.~L.~Ansorena~\cite[Theorem~4]{ansorena:2018} characterized
$\ell^1$-splicing bases in the following way: $(e_j)$ is $\ell^1$-splicing if and only if there is a
disjointly supported sequence $(x_j)$ in $X$ which is equivalent to the standard unit vector system
in $\ell^1$.  (Although this characterization is stated for weak$^*$ Schauder bases, examining the
proof reveals that it only uses~\eqref{eq:unconditional} and~\eqref{eq:spreading} and therefore
easily carries over to our topological basis.)  Examples of non-$\ell^1$-splicing weak$^*$ Schauder
bases are provided in~\cite{lechner:2019:subsymm} and~\cite{ansorena:2018}.


\section{Results}
\label{sec:results}

This section is divided into two parts: a subsection devoted to presenting our qualitative infinite
dimensional factorization results \Cref{thm:qual:1}, \Cref{thm:qual:2} and \Cref{thm:qual:3}, and a
subsection for the finite dimensional quantitative factorization results
\Cref{thm:fact:quant:uncond} and \Cref{cor:fact:quant:subs}.  The qualitative infinite dimensional
factorization results are achieved by building on the work done in~\cite{lechner:2019:subsymm} and
bringing in a randomization technique from~\cite{lechner:motakis:mueller:schlumprecht:2020:2}.

For related factorization results in various Banach spaces we refer to e.g.\@
\cite{casazza:kottman:lin:1977,capon:1982:1,bourgain:1983,blower:1990,wark:2007:class,wark:2007:direct,mueller:2012,lechner:mueller:2015,lechner:direct,lechner:2018:factor-mixed,lechner:2018:1-d,lechner:2018:2-d,lechner:motakis:mueller:schlumprecht:2020,lechner:motakis:mueller:schlumprecht:2019}
; see also~\cite{mueller:2005}.

\subsection{Qualitative Factorization results}
\label{sec:qual-fact-results}

First, we discuss the the qualitative infinite dimensional factorization results in Banach spaces
with a subsymmetric topological basis \Cref{thm:qual:1} and \Cref{thm:qual:2}, which establish that
subsymmetric Schauder bases and subsymmetric topological bases that are non-$\ell^1$-splicing have
the factorization property.  We then conclude this subsection with \Cref{thm:qual:3} asserting that
the array $(e_{n,j})_{n,j}$ corresponding to the direct sum of Banach spaces with a subsymmetric and
non-$\ell^1$-splicing topological basis also has the factorization property.  We would like to point
out that mixing of different Banach spaces in the direct sum is allowed, as long as the constants of
the subsymmetric topological basis are uniformly bounded.

\Cref{thm:qual:1} below can be viewed as a generalization of the factorization theorem
in~\cite{casazza:lin:1974} (see also~\cite[Proposition~3.b.8]{lindenstrauss:tzafriri:1977}).
\begin{thm}\label{thm:qual:1}
  Let $(e_j)$ be a $(C_u,C_s)$-subsymmetric Schauder basis for the Banach space $X$.  Then $(e_j)$
  has the $\bigl(\frac{2C_u^5C_s^3}{\delta}\bigr)$-factorization property.
\end{thm}

Where \Cref{thm:qual:1} demands that $(e_j)$ is a Schauder basis, \Cref{thm:qual:2} allows for more
general topological bases e.g.\@ a weak$^*$ Schauder bases.  To compensate for the lacking
norm-convergence of the topological basis, we demand that $(e_j)$ is non-$\ell^1$-splicing.

We are now ready to state our second main result.
\begin{thm}\label{thm:qual:2}
  Let the dual pair $(X,Y,\langle\cdot,\cdot\rangle)$ together with the sequences $(e_j)$, $(f_j)$
  satisfy \textrefp[B]{item:1}--\textrefp[B]{item:5} and assume that $(e_j)$ is
  $(C_u,C_s)$-subsymmetric as well as non-$\ell^1$-splicing.  Then $(e_j)$ has the
  $\bigl(\frac{2C_u^5C_s^3}{\delta}\bigr)$-factorization property.
\end{thm}
Recall that \textrefp[B]{item:1}--\textrefp[B]{item:5} are satisfied if $(e_j)$ is a weak$^*$
Schauder basis.  The proof of \Cref{thm:qual:2}, together with the proof of \Cref{thm:qual:1}, is
given in \Cref{sec:banach-spaces-with}.

In~\cite{casazza:kottman:lin:1977}, Casazza, Kottman and Lin showed that the spaces $\ell^p(X)$,
$1 < p < \infty$ and $c_0(X)$ are primary, whenever $X$ has a (sub)symmetric Schauder basis and is
not isomorphic to $\ell^1$.  Samuel~\cite{samuel:1978} proved that the spaces $\ell^p(\ell^q)$,
$1\leq p,q\leq\infty$ are primary.  Capon~\cite{capon:1982:1} showed that $\ell^1(X)$ and
$\ell^\infty(X)$ is primary, whenever $X$ has a (sub)symmetric Schauder basis.
In~\cite{lechner:2019:subsymm}, the author proved that the Banach spaces $\ell^p(X)$,
$1\leq p\leq\infty$ are primary, whenever $X$ is a Banach space with a non-$\ell^1$-splicing
subsymmetric weak$^*$ Schauder basis.

\Cref{thm:qual:3} below can be viewed as a vector valued version of \Cref{thm:qual:2}, but can also
be regarded as an extension of the factorization result~\cite[Theorem~1.2]{lechner:2019:subsymm} in
the sense that \cite[Theorem~1.2]{lechner:2019:subsymm} is a corollary to \Cref{thm:qual:3} (see
\Cref{cor:qual:3}).  This highlights the close connection between the factorization property of a
basis and the primarity of the space.  On the other hand, we would like to point out that although
the James space is primary~\cite{casazza:1977}, the boundedly complete basis of the James space does
not have the factorization
property~\cite[Proposition~2.5]{lechner:motakis:mueller:schlumprecht:2020}.
\begin{thm}\label{thm:qual:3}
  For each $n\in\mathbb{N}$ let the dual pair $(X_n,Y_n,\langle\cdot,\cdot\rangle)$ of infinite
  dimensional Banach spaces $X_n$ and $Y_n$ and the sequences $(e_{n,j})_{j}$, $(f_{n,j})_{j}$
  satisfy \textrefp[B]{item:1}--\textrefp[B]{item:5} with constant $C_d$ (uniformly $n$).  Assume
  that $(e_{n,j})_j$ is $(C_u,C_s)$-subsymmetric (uniformly in $n$) and non-$\ell^1$-splicing and
  let $1\leq p\leq \infty$.  Then $(e_{n,j})_{n,j}$ has the factorization property in
  $\ell^p((X_n))$, i.e.\@ whenever $T\colon \ell^p((X_n))\to \ell^p((X_n))$ is a bounded operator
  with
  \begin{equation*}
    \inf_{n,j}|\langle T e_{n,j}, f_{n,j} \rangle|
    > 0,
  \end{equation*}
  there exist bounded operators $E,P\colon \ell^p((X_n))\to \ell^p((X_n))$ such that
  $I_{\ell^p((X_n))} = PTE$.
\end{thm}
For the proof of \Cref{thm:qual:3} we refer to \Cref{sec:direct-sums-banach}.

\subsection{Finite dimensional quantitative factorization results}
\label{sec:finite-dimens-quant}

In~\cite[Theorem~6.1]{bourgain:tzafriri:1987}, Bourgain and Tzafriri obtain the following restricted
invertibility result for operators acting on $n$-dimensional Banach spaces with an unconditional
basis $(e_j)_{j=1}^n$ which satisfies a lower $r$-estimate for some $1 < r < \infty$: For any
$0 < \varepsilon < 1$, there exists a subset $\sigma$ with $|\sigma|\geq n^{1-\varepsilon}$ on which
a given operator is invertible when restricted to the subspace $[e_j : j\in\sigma]$.  Analyzing
their proof, we recognize that we can relax the condition that $(e_j)$ has to satisfy a lower
$r$-estimate (see \Cref{thm:fact:quant:uncond}), albeit at the cost of a weaker estimate for
$|\sigma|$ (see \Cref{rem:fact:quant:uncond}).  We would like to point out the closely related
recent works \cite{lechner:2018:1-d,lechner:2018:2-d}, in which the dependence on the dimension for
quantitative factorization results in one- and two-parameter Hardy and $\bmo$ spaces was improved
from super-exponential estimates to polynomial estimates.

In this subsection, $(e_j)$ denotes a normalized basis for the Banach space $X$ and $(e_j^*)$
denotes the biorthogonal functionals to $(e_j)$.  We define the function
$\tau\colon\mathbb{N}\to[0,\infty)$ by putting
\begin{equation}\label{eq:fundamentals:3}
  \tau(n)
  = \max\Bigl\{
  \min\Bigl(
  \max_{1\leq j\leq n} \Bigl\|\sum_{\substack{i=1\\i\neq j}}^n\varepsilon_{ij} e_i\Bigr\|_X,
  \max_{1\leq i\leq n} \Bigl\|\sum_{\substack{j=1\\j\neq i}}^n\varepsilon_{ij} e_j^*\Bigr\|_{X^*}
  \Bigr) : \varepsilon_{ij}\in\{\pm 1\},\ 1\leq i,j\leq n
  \Bigr\}.
\end{equation}
Note that if $(e_j)$ is $C_u$-unconditional, we have the estimates
\begin{equation}\label{eq:40}
  C_u^{-1}
  \leq \tau(n)
  \leq C_u \min\Bigl(
  \Bigl\|\sum_{i=1}^n e_i\Bigr\|_X,
  \Bigl\|\sum_{j=1}^n e_j^*\Bigr\|_{X^*}
  \Bigr),
  \qquad n\geq 2.
\end{equation}

We are now ready to state our first result on restricted invertibility.
\begin{thm}\label{thm:fact:quant:uncond}
  Let $(e_j)$ be a normalized $C_u$-unconditional basis for the Banach space $X$, let $(e_j^*)$
  denote the biorthogonal functionals to $(e_j)$ and put $X_n = [e_j : 1\leq j\leq n]$,
  $n\in\mathbb{N}$.  Let $n\in\mathbb{N}$ and $\delta,\Gamma,\eta > 0$ be such that
  \begin{equation}\label{eq:thm:fact:quant:1}
    \frac{\delta\min(1,\eta)}{4\Gamma n}
    \leq \tau(n)
    \leq \frac{\delta\min(1,\eta)}{2^{10}\Gamma}\cdot\frac{[16+\min(\eta,(1+\eta)^{-1})n]^2}{n}.
  \end{equation}
  Let $T\colon X_n\to X_n$ denote an operator satisfying
  \begin{equation}\label{eq:thm:fact:quant:2}
    \|T\|\leq \Gamma
    \qquad\text{and}\qquad
    |\langle T e_j, e_j^*\rangle|
    \geq \delta,
    \quad 1\leq j\leq n,
  \end{equation}
  and let $D\colon X_n\to X_n$ denote the diagonal operator of $T$, i.e.\@
  \begin{equation*}
    D e_i = \langle Te_i, e_i^*\rangle e_i,
    \qquad 1\leq i\leq n.
  \end{equation*}
  For each $\sigma\subset\{1,\ldots,n\}$ define the restriction operators
  $R_\sigma\colon X_n\to X_n$ by $R_\sigma(\sum_{i=1}^n a_i e_i) = \sum_{i\in\sigma} a_i e_i$.  Then
  there exists a subset $\sigma\subset\{1\ldots,n\}$ with
  \begin{equation}\label{eq:thm:fact:quant:3}
    |\sigma|
    \geq \sqrt{\frac{\delta \min(1,\eta)}{16\Gamma}}\cdot\sqrt{\frac{n}{\tau(n)}}
  \end{equation}
  such that the operator $R_\sigma D^{-1} T R_\sigma$ is invertible and satisfies
  \begin{equation}\label{eq:thm:fact:quant:4}
    \|(R_\sigma D^{-1} T R_\sigma)^{-1}\|\leq 1 + \eta.
  \end{equation}

  Moreover, if we define $X_\sigma = [e_j : j\in\sigma]$, there exist operators
  $E\colon X_\sigma\to X_n$ and $P\colon X_n\to X_\sigma$ with
  $\|E\|\|P\|\leq C_u^2\frac{1+\eta}{\delta}$ such that $I_{X_\sigma} = PTE$.
\end{thm}
\Cref{thm:fact:quant:uncond} will be proved in \Cref{sec:quant-fact-results}.

\begin{rem}\label{rem:fact:quant:uncond}
  We will now relate \Cref{thm:fact:quant:uncond} to \cite[Theorem~6.1]{bourgain:tzafriri:1987}.

  To this end, let $1 < r,s < \infty$ with $\frac{1}{r} + \frac{1}{s} = 1$.  Assume that
  $(e_j)_{j=1}^n$ satisfies a lower $r$-estimate with constant $c_r$, i.e.\@ there exists a constant
  $c_r > 0$ such that
  \begin{equation*}
    \Bigl\|\sum_{j=1}^n a_j e_j\Bigr\|
    \geq c_r \Bigl( \sum_{j=1}^n |a_j|^r \Bigr)^{1/r}
  \end{equation*}
  for all scalars $(a_i)$.  Then one can easily verify that $(e_j^*)_{j=1}^n$ satisfies an upper
  $s$-estimate with constant $\frac{1}{c_r}$, i.e.\@
  \begin{equation*}
    \Bigl\|\sum_{j=1}^n a_j e_j^*\Bigr\|
    \leq \frac{1}{c_r} \Bigl( \sum_{j=1}^n |a_j|^s \Bigr)^{1/s}.
  \end{equation*}
  In particular, we obtain the estimate $\tau(n)\leq\frac{1}{c_r} (n-1)^{1/s}$.  Thus, if we choose
  \begin{equation*}
    n\geq \Bigl(\frac{2^{10}\Gamma}{c_r\delta\min(1,\eta)\min(\eta, (1+\eta)^{-1})}\Bigr)^{1/r},
  \end{equation*}
  \Cref{thm:fact:quant:uncond} yields a subset $\sigma\subset\{1,\ldots,n\}$ with
  \begin{equation*}
    |\sigma|
    \geq \sqrt{\frac{c_r \delta \min(1,\eta)}{16\Gamma}}\cdot n^{1/(2r)}.
  \end{equation*}
\end{rem}

In~\cite[Corollary~4.4]{bourgain:tzafriri:1989} Bourgain and Tzafriri obtained linear lower
estimates for the set $|\sigma|$, whenever the subsymmetric basis $(e_j)$ satisfies certain
conditions in terms of Boyd indices.  Using \Cref{thm:fact:quant:uncond}, we can get rid of this
restriction entirely, but the prize we pay is a weaker estimate for $|\sigma|$ (see
\Cref{cor:fact:quant:subs}, below).
\begin{cor}\label{cor:fact:quant:subs}
  Let $(e_j)$ be a normalized $(C_u,C_s)$-subsymmetric basis for the Banach space $X$, let $(e_j^*)$
  denote the biorthogonal functionals to $(e_j)$ and put $X_n = [e_j : 1\leq j\leq n]$,
  $n\in\mathbb{N}$.  Let $n\in\mathbb{N}$ and $\delta,\Gamma,\eta > 0$ be such that
  \begin{equation}\label{eq:62}
    n\geq 1 + \frac{C_u \delta \min(1,\eta)}{4\Gamma}
    \qquad\text{and}\qquad
    n\geq \frac{2^{11}C_s^3}{\delta^2\min(1,\eta^2)\min(\eta^4,(1+\eta)^{-4})}.
  \end{equation}
  Let $T\colon X_n\to X_n$ denote an operator satisfying
  \begin{equation}\label{eq:cor:fact:quant:subs:2}
    \|T\|\leq \Gamma
    \qquad\text{and}\qquad
    |\langle T e_j, e_j^*\rangle|
    \geq \delta,
    \quad 1\leq j\leq n.
  \end{equation}
  As in \Cref{thm:fact:quant:uncond}, $D\colon X_n\to X_n$ denotes the diagonal operator of $T$ and
  $R_\sigma\colon X_n\to X_n$, $\sigma\subset\{1,\ldots,n\}$ the restriction operators.  Then there
  exists a subset $\sigma\subset\{1\ldots,n\}$ with
  \begin{equation}\label{eq:63}
    |\sigma|
    \geq 4\cdot
    (2 C_u^3C_s^3)^{-1/4}\cdot
    \sqrt{\frac{\delta\min(1,\eta)}{\Gamma}}\cdot
    n^{1/4}
  \end{equation}
  such that the operator $R_\sigma D^{-1} T R_\sigma$ is invertible and satisfies
  \begin{equation}\label{eq:64}
    \|(R_\sigma D^{-1} T R_\sigma)^{-1}\|\leq 1 + \eta.
  \end{equation}

  Thus, for $k = |\sigma|$, there exist operators $E\colon X_k\to X_n$ and $P\colon X_n\to X_k$ with
  $\|E\|\|P\|\leq C_u^2C_s^2\frac{1+\eta}{\delta}$ such that $I_{X_k} = PTE$.
\end{cor}
The proof of \Cref{cor:fact:quant:subs} can be found in \Cref{sec:quant-fact-results}.


\section{Tools}
\label{sec:tools}

Here, we provide the mathematical tools that will be used throughout this article.  They comprise of
subspace annihilation results (see \Cref{lem:annihil}, \Cref{lem:splicing}) and techniques by which
we preserve the large diagonal of an operator under blocking of the basis (see \Cref{lem:diag} and
\Cref{pro:subsymm}).  Moreover, we provide estimates for basic factorization operators that are
obtained by blocking a subsymmetric basis (see \Cref{pro:emb-proj}).

\begin{lem}\label{lem:annihil}
  Let $\mathcal{I}\subset\mathbb{N}$ denote an infinite set, let $n\in\mathbb{N}$,
  $L\in2\cdot\mathbb{N}$ and $\eta > 0$.  Then for all bounded sequences $(x_j)$ in $X$ and $(y_j)$
  in $Y$ and every $n\in\mathbb{N}$ there exists an infinite set $\Lambda\subset\mathcal{I}$ such
  that
  \begin{align*}
    \sup \Bigl\{\Big|\Big\langle\sum_{k\in \mathcal{B}} \varepsilon_k x_k, y_j \Big\rangle\Big| :
    \mathcal{B}\subset\Lambda,\ |\mathcal{B}| = L,\
    \varepsilon\in \mathcal{E}(\mathcal{B}),\ 1\leq j\leq n\Bigr\}
    &\leq \eta,\\
    \sup \Bigl\{\Big|\Big\langle x_j, \sum_{k\in \mathcal{B}} \varepsilon_k y_k\Big\rangle\Big| :
    \mathcal{B}\subset\Lambda,\ |\mathcal{B}| = L,\
    \varepsilon\in \mathcal{E}(\mathcal{B}),\ 1\leq j\leq n\Bigr\}
    &\leq \eta,
  \end{align*}
  where the set $\mathcal{E}(\mathcal{B})$ is given by
  \begin{equation}\label{eq:1}
    \mathcal{E}(\mathcal{B})
    = \Big\{ (\varepsilon_k)\in\{\pm 1\}^{\mathcal{B}} :
    \sum_{k\in \mathcal{B}} \varepsilon_k = 0 \Big\}.
  \end{equation}
\end{lem}

\begin{proof}
  The proof is a straightforward adaptation of the argument given for
  \cite[Lemma~3.1]{lechner:2019:subsymm}.  For sake of completeness, we will give a short proof,
  here.

  First, we put
  $\mathcal{G}_k^1 = \bigl\{i\in\mathcal{I} : \frac{(k-1)\eta}{L} < \langle x_i, y_1\rangle\leq
  \frac{k\eta}{L} \bigr\}$ and note that by \textrefp[B]{item:3}
  $\bigcup_{k\in\mathbb{Z}}\mathcal{G}_k^1 = \mathcal{I}$ and that there are only finitely many
  non-empty $\mathcal{G}_k^1$, $k\in\mathbb{Z}$.  Since $\mathcal{I}$ is infinite, there exists at
  least one $k_1\in\mathbb{Z}$ such that $\mathcal{G}_{k_1}^1$ is also infinite.  Next, we define
  $\mathcal{G}_k^2 = \bigl\{i\in\mathcal{G}_{k_1}^1 : \frac{(k-1)\eta}{L} < \langle x_i,
  y_2\rangle\leq \frac{k\eta}{L} \bigr\}$ and repeat the previous step to obtain a
  $k_2\in\mathbb{Z}$ such that $\mathcal{G}_{k_2}^2$ is infinite.  After $n$ steps, we obtain an
  infinite set $\mathcal{G} = \mathcal{G}_{k_n}^n$ such that
  \begin{equation}\label{eq:2}
    \sup \bigl\{
    |\langle x_{i_0}, y_j\rangle - \langle x_{i_1}, y_j\rangle|
    : i_0, i_1\in\mathcal{G},\ 1\leq j\leq n
    \bigr\}
    \leq \frac{2\eta}{L}.
  \end{equation}

  We will now repeat the above process but with the roles of $x$ and $y$ reversed and with
  $\mathcal{G}$ instead of $\mathcal{I}$.  To illustrate, we now put
  $\mathcal{H}_k^1 = \{i\in\mathcal{G} : \frac{(k-1)\eta}{L} < \langle x_1, y_i\rangle\leq
  \frac{k\eta}{L} \}$.  With the same reasoning as above, we can find $l_1\in\mathbb{Z}$ such that
  $\mathcal{H}_{l_1}^1$ is infinite.  Iterating this procedure and stopping after $n$ steps yields
  an infinite set $\Lambda = \mathcal{H}_{l_n}^n\subset\mathcal{G}$ such that
  \begin{equation}\label{eq:3}
    \sup \bigl\{
    |\langle x_j, y_{i_0}\rangle - \langle x_j, y_{i_1}\rangle|
    : i_0, i_1\in\Lambda,\ 1\leq j\leq n
    \bigr\}
    \leq \frac{2\eta}{L}.
  \end{equation}
  Note that for each $\mathcal{B}\subset\Lambda$ with $|\mathcal{B}|=L$ and each
  $(\varepsilon_k)\in\mathcal{E}(\mathcal{B})$, there are exactly $L/2$ $k\in\mathcal{B}$ such that
  $\varepsilon_k=1$ and $L/2$ $k\in\mathcal{B}$ such that $\varepsilon_k=-1$.  This observation
  together with~\eqref{eq:2} and~\eqref{eq:3} proves the assertion.
\end{proof}

For each $L\in 2\cdot\mathbb{N}$ and $\mathcal{A}\subset\mathbb{N}$ with
$L\leq |\mathcal{A}| < \infty$, we define the finite set $\Omega_L^{\mathcal{A}}$ by
\begin{equation}\label{eq:4}
  \Omega_L^{\mathcal{A}}
  = \{ (\mathcal{B}, (\varepsilon_k)) : \mathcal{B}\subset\mathcal{A},\ |\mathcal{B}|=L,\ 
  (\varepsilon_k)\in\mathcal{E}(\mathcal{B})\}.
\end{equation}
By $\cond_L^{\mathcal{A}}$ we denote the average over all elements in $\Omega_L^{\mathcal{A}}$.  For
each $(\mathcal{B},(\varepsilon_k))\in\Omega_L^{\mathcal{A}}$, we define
\begin{equation}
  \label{eq:5}
  b_{\mathcal{B}}^{(\varepsilon_k)}
  = \sum_{k\in\mathcal{B}} \varepsilon_k e_k
  \qquad\text{and}\qquad
  d_{\mathcal{B}}^{(\varepsilon_k)}
  = \sum_{k\in\mathcal{B}} \varepsilon_k f_k.
\end{equation}
Before turning to the next Lemma, we define $\nu\colon\mathbb{N}\to [0,\infty)$ by
\begin{equation}\label{eq:fundamentals:1}
  \nu(n)
  = \sup\Bigl\{
  \min\Bigl(
  \max_{l\in\mathcal{A}} \Bigl\|\sum_{k\in\mathcal{A}\setminus\{l\}}e_k\Bigr\|_X,
  \max_{k\in\mathcal{A}} \Bigl\|\sum_{l\in\mathcal{A}\setminus\{k\}}f_l\Bigr\|_Y
  \Bigr) :
  \mathcal{A}\subset\mathbb{N},\ |\mathcal{A}| = n
  \Bigr\}.
\end{equation}

The following lemma uses a randomization technique
from~\cite{lechner:motakis:mueller:schlumprecht:2020:2}.  In contrast to the randomization over all
possible choices of signs in~\cite{lechner:motakis:mueller:schlumprecht:2020:2}, we average in
\Cref{lem:diag} over all signs in $\mathcal{E}(\mathcal{B})$.  This enables us to use
\Cref{lem:annihil} which is essential for the non-separable case (see \Cref{thm:qual:2}), but has
the drawback that it introduces a negative bias when averaging $\varepsilon_k\varepsilon_l$,
$k\neq l$ (see~\eqref{eq:7}).  This bias can be traced to the term $A_2$ in~\eqref{eq:8} and that in
turn, by averaging over possible choices of $\mathcal{B}\subset\mathcal{A}$, leads to the term $B_2$
in~\eqref{eq:11}, where it is finally controlled in terms of $\nu$.
\begin{lem}\label{lem:diag}
  Let the dual pair $(X,Y,\langle\cdot,\cdot\rangle)$ together with the sequences $(e_j)$, $(f_j)$
  satisfy \textrefp[B]{item:1}--\textrefp[B]{item:5}.  Let $T\colon X\to X$ denote a bounded linear
  operator such that $\delta := \inf_j \langle Te_j, f_j\rangle > 0$.  Let $L\in 2\cdot\mathbb{N}$,
  $N\in\mathbb{N}$ with $N\geq L$ and pick any $\mathcal{A}\subset\mathbb{N}$ with
  $|\mathcal{A}|=N$.  Then
  \begin{equation*}
    \cond_L^{\mathcal{A}} \langle T b_{\mathcal{B}}^{(\varepsilon_k)},
    d_{\mathcal{B}}^{(\varepsilon_k)}\rangle
    \geq \Bigl[
    \delta - C_d\frac{\|T\|}{N-1}\nu(N)
    \Bigr]\cdot L.
  \end{equation*}
\end{lem}

\begin{proof}
  Let $\mathcal{B}\subset\mathcal{A}$ with $|\mathcal{B}| = L$ be fixed and note that
  by~\eqref{eq:5}, \eqref{eq:4} and~\eqref{eq:1}
  \begin{equation}\label{eq:6}
    \frac{1}{\binom{L}{L/2}}\sum_{(\varepsilon_k)\in\mathcal{E}(\mathcal{B})}
    \langle T b_{\mathcal{B}}^{(\varepsilon_k)}, d_{\mathcal{B}}^{(\varepsilon_k)}\rangle
    = \sum_{k\in\mathcal{B}} \langle T e_k, f_k\rangle
    + \sum_{k\neq l\in\mathcal{B}}
    \frac{1}{\binom{L}{L/2}}\sum_{(\varepsilon_k)\in\mathcal{E}(\mathcal{B})}
    \varepsilon_k\varepsilon_l \langle T e_k, f_l\rangle.
  \end{equation}
  A straightforward calculation shows that
  \begin{equation}\label{eq:7}
    \frac{1}{\binom{L}{L/2}}\sum_{(\varepsilon_k)\in\mathcal{E}(\mathcal{B})}
    \varepsilon_k\varepsilon_l
    = \frac{-1}{L-1},
    \qquad k\neq l.
  \end{equation}
  Combining~\eqref{eq:6} with~\eqref{eq:7} yields
  \begin{equation}
    \label{eq:8}
    \frac{1}{\binom{L}{L/2}}\sum_{(\varepsilon_k)\in\mathcal{E}(\mathcal{B})}
    \langle T b_{\mathcal{B}}^{(\varepsilon_k)}, d_{\mathcal{B}}^{(\varepsilon_k)}\rangle
    = \sum_{k\in\mathcal{B}} \langle T e_k, f_k\rangle
    - \frac{1}{L-1}\sum_{k\neq l\in\mathcal{B}} \langle T e_k, f_l\rangle
    = A_1 - \frac{1}{L-1} A_2.
  \end{equation}

  We will now separately average $A_1$ and $A_2$ over all $\binom{N}{L}$ possible selections
  $\mathcal{B}\subset\mathcal{A}$ with $|\mathcal{B}|=L$.  First, averaging $A_1$ yields
  \begin{align*}
    \frac{1}{\binom{N}{L}}\sum_{\substack{\mathcal{B}\subset\mathcal{A}\\|\mathcal{B}|=L}} A_1
    &= \frac{1}{\binom{N}{L}}\sum_{\substack{\mathcal{B}\subset\mathcal{A}\\|\mathcal{B}|=L}}
    \sum_{k\in\mathcal{B}} \langle T e_k, f_k\rangle
    = \sum_{k\in\mathcal{A}} \langle T e_k, f_k\rangle
    \frac{|\{\mathcal{B}\subset A : \mathcal{B}\ni k,\ |\mathcal{B}|=L\}|}{\binom{N}{L}}\\
    &= \frac{\binom{N-1}{L-1}}{\binom{N}{L}} \sum_{k\in\mathcal{A}} \langle T e_k, f_k\rangle
  \end{align*}
  and we record
  \begin{equation}
    \label{eq:9}
    \frac{1}{\binom{N}{L}}\sum_{\substack{\mathcal{B}\subset\mathcal{A}\\|\mathcal{B}|=L}} A_1
    = \frac{L}{N} \sum_{k\in\mathcal{A}} \langle T e_k, f_k\rangle.
  \end{equation}
  Secondly, averaging $A_2$ gives us
  \begin{align*}
    \frac{1}{\binom{N}{L}}\sum_{\substack{\mathcal{B}\subset\mathcal{A}\\|\mathcal{B}|=L}} A_2
    &= \frac{1}{\binom{N}{L}}\sum_{\substack{\mathcal{B}\subset\mathcal{A}\\|\mathcal{B}|=L}}
    \sum_{k\neq l\in\mathcal{B}} \langle T e_k, f_l\rangle
    = \sum_{k\neq l\in\mathcal{A}} \langle T e_k, f_l\rangle
    \frac{|\{\mathcal{B}\subset A : \mathcal{B}\ni k,l,\ |\mathcal{B}|=L\}|}{\binom{N}{L}}\\
    &= \frac{\binom{N-2}{L-2}}{\binom{N}{L}} \sum_{k\neq l\in\mathcal{A}} \langle T e_k, f_l\rangle,
  \end{align*}
  from which immediately follows that
  \begin{equation}
    \label{eq:10}
    \frac{1}{\binom{N}{L}}\sum_{\substack{\mathcal{B}\subset\mathcal{A}\\|\mathcal{B}|=L}} A_2
    = \frac{L(L-1)}{N(N-1)} \sum_{k\neq l\in\mathcal{A}} \langle T e_k, f_l\rangle.
  \end{equation}
  Combining~\eqref{eq:8}, \eqref{eq:9} and~\eqref{eq:10} yields
  \begin{equation}
    \label{eq:11}
    \cond_L^{\mathcal{A}}
    \langle T b_{\mathcal{B}}^{(\varepsilon_k)}, d_{\mathcal{B}}^{(\varepsilon_k)}\rangle
    = \frac{L}{N} \sum_{k\in\mathcal{A}} \langle T e_k, f_k\rangle
    - \frac{L}{N(N-1)} \sum_{k\neq l\in\mathcal{A}} \langle T e_k, f_l\rangle
    = \frac{L}{N} B_1 - \frac{L}{N(N-1)} B_2.
  \end{equation}

  By definition of $\delta$ and $\mathcal{A}$, we obtain
  \begin{equation}\label{eq:12}
    B_1
    = \sum_{k\in\mathcal{A}} \langle T e_k, f_k\rangle
    \geq \delta N.
  \end{equation}
  Now we will estimate $B_2$ in two different ways, each exploiting the linearity of
  $\langle\cdot, \cdot\rangle$ in each component.  Exploiting the linearity in the second component
  of the bilinear form and using \textrefp[B]{item:3} yields
  \begin{align*}
    B_2
    &= \sum_{k\neq l\in\mathcal{A}} \langle T e_k, f_l\rangle
      = \sum_{k\in\mathcal{A}}
      \Bigl\langle T e_k, \sum_{l\in\mathcal{A}\setminus\{k\}} f_l\Bigr\rangle
      \leq C_d \|T\| \sum_{k\in\mathcal{A}} \Bigl\|\sum_{l\in\mathcal{A}\setminus\{k\}} f_l\Bigr\|\\
    &\leq C_d \|T\| N \max_{k\in\mathcal{A}} \Bigl\|\sum_{l\in\mathcal{A}\setminus\{k\}} f_l\Bigr\|
      = C_d \|T\| N \max_{k\in\mathcal{A}} \Bigl\|\sum_{l\in\mathcal{A}\setminus\{k\}}f_l\Bigr\|_Y.
  \end{align*}
  Using the exact same steps as above but exploiting the linearity in the first component of the
  bilinear form yields
  \begin{equation*}
    B_2
    \leq C_d \|T\| N \max_{l\in\mathcal{A}} \Bigl\|\sum_{k\in\mathcal{A}\setminus\{l\}}e_k\Bigr\|_X.
  \end{equation*}
  Thus, we have the estimate
  \begin{equation}\label{eq:13}
    B_2
    \leq C_d \|T\| N \nu(N).
  \end{equation}
  Combining~\eqref{eq:11}, \eqref{eq:12} and~\eqref{eq:13} concludes the proof.
\end{proof}

For $n\in\mathbb{N}$ we define the functions $\lambda,\mu\colon \mathbb{N}\to [0,\infty)$ by
\begin{equation}\label{eq:fundamentals:2}
  \lambda(n)
  = \Bigl\| \sum_{j=1}^n e_j\Bigr\|_X
  \qquad\text{and}\qquad
  \mu(n)
  = \Bigl\| \sum_{j=1}^n f_j\Bigr\|_Y.
\end{equation}
If $(e_j)$ is $C_s$-spreading, then by~\eqref{eq:fundamentals:1} and~\eqref{eq:spreading}, we obtain
\begin{equation}
  \label{eq:58}
  \nu(n)\leq C_s\min(\lambda(n-1), \mu(n-1)),
  \qquad n\in\mathbb{N}.
\end{equation}
We note that keeping track of the constants in the proof of Proposition~3.a.6 and Proposition~3.a.4
in \cite{lindenstrauss:tzafriri:1977} yields
\begin{equation}\label{eq:14}
  \lambda(n)\mu(n)\leq 2 C_uC_s n
\end{equation}

We are now ready to prove \Cref{pro:subsymm} by specializing \Cref{lem:diag} to the case where
$(e_j)$ is subsymmetric.
\begin{pro}\label{pro:subsymm}
  Let the dual pair $(X,Y,\langle\cdot,\cdot\rangle)$ together with the sequences $(e_j)$, $(f_j)$
  satisfy \textrefp[B]{item:1}--\textrefp[B]{item:5} and assume that $(e_j)$ is
  $(C_u,C_s)$-subsymmetric.  Let $T\colon X\to X$ denote a bounded linear operator such that
  $\delta := \inf_j \langle Te_j, f_j\rangle > 0$.  For $L\in 2\cdot\mathbb{N}$, $0 < \kappa < 1$ we
  define
  \begin{equation}\label{eq:15}
    N = \max \Bigl(L, 1 + \Bigl\lceil \frac{2 C_d^2 C_u C_s^3 \|T\|^2}{\kappa^2 \delta^2} \Bigr\rceil\Bigr).
  \end{equation}
  Then for each $\mathcal{A}\subset\mathbb{N}$ with $|\mathcal{A}|=N$ we have
  \begin{equation}\label{eq:16}
    \cond_L^{\mathcal{A}} \langle T b_{\mathcal{B}}^{(\varepsilon_k)},
    d_{\mathcal{B}}^{(\varepsilon_k)}\rangle
    \geq (1-\kappa)\delta L.
  \end{equation}
  In particular, there exists a set $\mathcal{B}\subset\mathcal{A}$ with $|\mathcal{B}|=L$ and a
  choice of signs $(\varepsilon_k)\in\mathcal{E}(\mathcal{B})$ such that
  \begin{equation}\label{eq:17}
    \Bigl\langle
    T \sum_{k\in\mathcal{B}} \varepsilon_k e_k,
    \sum_{k\in\mathcal{B}} \varepsilon_k f_k
    \Bigr\rangle
    \geq (1-\kappa)\delta L.
  \end{equation}
\end{pro}

\begin{proof}
  \Cref{lem:diag} yields
  \begin{equation*}
    \cond_L^{\mathcal{A}} \langle T b_{\mathcal{B}}^{(\varepsilon_k)},
    d_{\mathcal{B}}^{(\varepsilon_k)}\rangle
    \geq \Bigl[
    \delta - C_d\frac{\|T\|}{N-1}\nu(N)
    \Bigr]\cdot L.
  \end{equation*}
  Using~\eqref{eq:58} and~\eqref{eq:14} gives us
  \begin{align*}
    \nu(N)
    &\leq C_s\min(\lambda(N-1), \mu(N-1))
      \leq C_s\sqrt{\lambda(N-1)\mu(N-1)}\\
    &\leq \sqrt{2C_uC_s^3(N-1)}.
  \end{align*}
  Thus far, we proved
  \begin{equation*}
    \cond_L^{\mathcal{A}} \langle T b_{\mathcal{B}}^{(\varepsilon_k)},
    d_{\mathcal{B}}^{(\varepsilon_k)}\rangle
    \geq \Bigl[
    \delta - C_d\frac{\|T\|}{\sqrt{N-1}}\cdot
    \sqrt{2C_uC_s^3}
    \Bigr]\cdot L.
  \end{equation*}
  The latter inequality together with~\eqref{eq:15} implies~\eqref{eq:16}.  \eqref{eq:17} directly
  follows from~\eqref{eq:16} and definition of $\cond_L^{\mathcal{A}}$.
\end{proof}

We will now restate the result~\cite[Lemma~3.2]{lechner:2019:subsymm} (which concerns weak$^*$
Schauder bases) for our dual system.
\begin{lem}\label{lem:splicing}
  Let the dual pair $(X,Y,\langle\cdot,\cdot\rangle)$ together with the sequences $(e_j)$, $(f_j)$
  satisfy \textrefp[B]{item:1}--\textrefp[B]{item:5} and assume that $(e_j)_{j=1}^\infty$ is an
  unconditional non-$\ell^1$-splicing sequence.  Let $\mathcal{I}\subset\mathbb{N}$ denote an
  infinite set, $\eta > 0$ and $y_1,\ldots,y_n\in Y$ and let $T\colon X\to X$ denote a bounded
  linear operator.  Then there exists an infinite set $\Lambda\subset\mathcal{I}$ such that
  \begin{equation*}
    \sup_{\|x\|_{X}\leq 1} |\langle T P_{\Lambda} x, y_j \rangle| \leq \eta,
    \qquad 1\leq j\leq n.
  \end{equation*}
\end{lem}

\begin{proof}
  The same proof given for~\cite[Lemma~3.2]{lechner:2019:subsymm} applies to our case for $n=1$ and
  is therefore omitted.  Thus we may assume that \Cref{lem:splicing} has already been established
  for $n=1$.

  We will now inductively apply the Lemma.  In the first step, we apply the Lemma with $n=1$ to
  $y_1$ and obtain an infinite set $\Lambda_1\subset\mathcal{I}$ such that
  $\sup_{\|x\|_{X}\leq 1} |\langle TP_{\Lambda} x, y_1 \rangle| \leq \eta/C_u$.  Next, we apply the
  Lemma with $n=1$ to $\mathcal{I}=\Lambda_1$ and $y_2$ and obtain an infinite set
  $\Lambda_2\subset\Lambda_1$ such that
  $\sup_{\|x\|_{X}\leq 1} |\langle TP_{\Lambda_2} x, y_2 \rangle| \leq \eta/C_u$.  Continuing in
  this manner and stopping after $n$ steps yields infinite sets
  $\Lambda_1\supset\Lambda_2\supset\cdots\supset\Lambda_n$ such that
  $\sup_{\|x\|_{X}\leq 1} |\langle TP_{\Lambda_j} x, y_j \rangle| \leq \eta/C_u$, $1\leq j\leq n$.
  Defining $\Lambda = \Lambda_n$ and observing that for $x\in X$ and $1\leq j\leq n$,
  unconditionality yields that
  \begin{equation*}
    |\langle TP_{\Lambda} x, y_j \rangle|
    = |\langle TP_{\Lambda_j}P_{\Lambda} x, y_j \rangle|
    \leq \eta \|P_{\Lambda} x\|/C_u
    \leq \eta \|x\|
  \end{equation*}
  as claimed.
\end{proof}

The following proposition estimates our basic factorization operators $B,Q$ acting on a subsymmetric
sequence $(e_j)$.
\begin{pro}\label{pro:emb-proj}
  Let the dual pair $(X,Y,\langle\cdot,\cdot\rangle)$ together with the sequences $(e_j)$, $(f_j)$
  satisfy \textrefp[B]{item:1}--\textrefp[B]{item:5} and assume that $(e_j)_{j=1}^\infty$ is a
  $(C_u,C_s)$-subsymmetric sequence.  Let $L\in\mathbb{N}$ and let $(\mathcal{B}_j)$ denote a
  sequence of sets $\mathcal{B}_j\subset\mathbb{N}$ with $|\mathcal{B}_j| = L$,
  $\mathcal{B}_j < \mathcal{B}_{j+1}$, $j\in\mathbb{N}$, and let
  $(\varepsilon_k)\in\{\pm 1\}^{\mathbb{N}}$.  Define
  \begin{equation}\label{eq:18}
    b_j
    = \sum_{k\in\mathcal{B}_j} \varepsilon_k e_k
    \qquad\text{and}\qquad
    d_j
    = \sum_{k\in\mathcal{B}_j} \varepsilon_k f_k,
  \end{equation}
  as well as the operators $B,Q\colon X\to X$ by
  \begin{equation}\label{eq:BQ}
    Bx
    = \sum_{j=1}^\infty \langle x, f_j\rangle b_j
    \quad\text{and}\quad
    Qx
    = \sum_{j=1}^\infty \langle x, d_j\rangle e_j,
    \qquad x\in X.
  \end{equation}
  The operators are well defined, $QB = L\cdot I_{X}$ and we have the estimates
  \begin{equation}\label{eq:19}
    \|B\|, \|Q\|
    \leq C_uC_s\cdot L.
  \end{equation}
\end{pro}

\begin{proof}
  The assertions follow by essentially repeating Step~2 in the proof given for
  \cite[Theorem~1.1]{lechner:2019:subsymm}.  Since the proof is short we include it here.

  First we pick any $(n_j^l)$, such that $\{n_j^l : 1\leq l\leq L\} = \mathcal{B}_j$,
  $j\in\mathbb{N}$ and note that the sequence $(n_j^l)_j$ is increasing for each $1\leq l\leq L$.
  Now, let $x = \sum_{j=1}^\infty a_je_j$ be a series that converges in the $\sigma(X,Y)$-topology.
  Hence, by~\eqref{eq:unconditional} and~\eqref{eq:spreading}
  $Bx = \sum_{l=1}^L\sum_{j=1}^\infty \varepsilon_{n_j} a_j e_{n_j^l}$ is well defined and
  satisfies~\eqref{eq:19} as claimed.  Similarly, we see that since
  $Q x = \sum_{l=1}^L\sum_{j=1}^\infty \varepsilon_{n_j^l} a_{n_j^l} e_j$, the operator $Q$ is well
  defined and satisfies~\eqref{eq:19}.  The identity $QB = L\cdot I_X$ promptly follows from
  \textrefp[B]{item:4}.
\end{proof}


\section{Banach spaces with a subsymmetric basis}
\label{sec:banach-spaces-with}

In this section, we will prove our first two main results \Cref{thm:qual:1} and \Cref{thm:qual:2}.
Due to the big overlap of the arguments for \Cref{thm:qual:1} and \Cref{thm:qual:2}, we will show
them both simultaneously by bifurcating the proof in three places.  This bifurcation is visualized
by two color coded columns that will appear side by side.  The proofs for \Cref{thm:qual:1} and
\Cref{thm:qual:2} are obtained by reading the text on the white and light teal background,
respectively the text on the white and light magenta background.
\begin{myproof}[Proof of \Cref{thm:qual:1} and \Cref{thm:qual:2}]
  In the case of \Cref{thm:qual:1}, we put $f_j = e_j^*$, $1\leq j\leq n$, where $(e_j^*)$ denotes
  the sequence of biorthogonal functionals to $(e_j)$.
  
  Let $T\colon X\to X$ be a bounded linear operator which has $\delta$-large diagonal with respect
  to $(e_j)$.  Given $\eta > 0$, we want to show that $I_X$
  $\bigl(\frac{2C_u^5C_s^3}{\delta} + \eta\bigr)$-factors through $T$.  First, we define the
  multiplier $M\colon X\to X$ by $M e_j = \sign(\langle T e_j, f_j \rangle) e_j$ and note that
  by~\eqref{eq:unconditional} $M$ is a well defined operator satisfying $\|M\|\leq C_u$.  Thus, if
  we define $\widetilde T = TM$, then
  \begin{equation}\label{eq:20}
    \inf_j \langle \widetilde T e_j, f_j\rangle
    = \inf_j |\langle T e_j, f_j\rangle|
    = \delta
    > 0.
  \end{equation}

  Our strategy is to use \Cref{lem:annihil} (and additionally \Cref{lem:splicing} for the proof of
  \Cref{thm:qual:2}) to diagonalize the operator with a block basis $(b_j)$ of $(e_j)$, and
  \Cref{pro:subsymm} to transfer the large diagonal that $\widetilde{T}$ has with respect to $(e_j)$
  over to the block basis $(b_j)$.

  Throughout this proof, we will use the following constants:
  \begin{equation}\label{eq:21}
    \kappa = \frac{1}{2+\frac{4C_u^5C_s^3}{\eta\delta}},
    \quad
    \eta_i = \frac{2^{-i-1}}{C_di}\kappa\delta,\ i\in \mathbb{N},
    \quad L = 2,
    \quad
    N = \max \Bigl(2, 1 + \Bigl\lceil \frac{2 C_u C_s^3 \|T\|^2}{\kappa^2 \delta^2} \Bigr\rceil\Bigr).
  \end{equation}
  We note that if $f_j = e_j^*$, $j\in\mathbb{N}$, then $C_d=1$ (see \textrefp[B]{item:3}).

  \begin{proofstep}[Step~\theproofstep: Construction of the block basis]
    In each step of the subsequent construction, we will employ \Cref{lem:annihil},
    \Cref{lem:splicing} and \Cref{pro:subsymm} to $\widetilde{T}$ and with the constants $\kappa$,
    $L$ and $N$ as defined in~\eqref{eq:21}.  Other parameters will be explicitly specified at the
    appropriate place in the proof.

    For our initial step, we use \Cref{pro:subsymm} with
    $\mathcal{A} = \mathcal{A}_1 = \{1,\ldots,N\}$ to obtain a set
    $\mathcal{B}_1\subset\mathcal{A}_1$ with $|\mathcal{B}_1|=2$ and signs
    $(\varepsilon_k)\in\mathcal{E}(\mathcal{B}_1)$ such that
    \begin{equation}\label{eq:22}
      \langle \widetilde{T} b_1, d_1\rangle
      \geq 2(1-\kappa)\delta,
    \end{equation}
    where we defined
    \begin{equation}\label{eq:23}
      b_1
      = \sum_{k\in\mathcal{B}_1} \varepsilon_k e_k
      \qquad\text{and}\qquad
      d_1
      = \sum_{k\in\mathcal{B}_1} \varepsilon_k f_k.
    \end{equation}
    This completes the initial step of our construction.

    Assume that we have already chosen pairwise disjoint sets
    $\mathcal{B}_1 < \mathcal{B}_2 < \cdots < \mathcal{B}_{i-1}$ with $|\mathcal{B}_j|=2$,
    $1\leq j \leq i-1$, selected signs $(\varepsilon_k)\in\mathcal{E}(\mathcal{B}_j)$,
    $1\leq j \leq i-1$ and that we have defined
    \begin{equation}\label{eq:24}
      b_j
      = \sum_{k\in\mathcal{B}_j} \varepsilon_k e_k
      \quad\text{and}\quad
      d_j
      = \sum_{k\in\mathcal{B}_j} \varepsilon_k f_k,
      \qquad 1\leq j \leq i-1.
    \end{equation}
    We will now construct $b_i$ and $d_i$.
    \begin{mdframed}[backgroundcolor=black!5,leftline=false,rightline=false,topline=false,bottomline=false]
      \begin{minipage}[t]{.5\linewidth}
        \begin{mdframed}[backgroundcolor=cyan!5]
          \begin{mdframed}[backgroundcolor=gray!20]
            Schauder basis (\Cref{thm:qual:1}).
          \end{mdframed}
          \hfill\newline In this case, we first use \Cref{lem:annihil} with
          \begin{align*}
            \mathcal{I}
            = \{ k\in\mathbb{N} : k > \max \mathcal{B}_{i-1}\},\\
            \eta = \eta_i,\quad
            x_j = \widetilde{T} e_j,\quad
            y_j = d_j
          \end{align*}
          for all $1\leq j\leq i-1$ to obtain an infinite set $\Lambda_i^0\subset \mathcal{I}$ such
          that
          \begin{equation}\label{eq:25}
            \Big|\Big\langle
            \widetilde{T} \sum_{k\in \mathcal{B}} \varepsilon_k e_k,
            d_j
            \Big\rangle\Big|
            \leq \eta_i
          \end{equation}
          for all $\mathcal{B}\subset\Lambda_i^0$ with $|\mathcal{B}| = 2$, all
          $(\varepsilon_k)\in \mathcal{E}(\mathcal{B})$ and $1\leq j\leq i-1$.
        \end{mdframed}
      \end{minipage}
      \begin{minipage}[t]{.5\linewidth}
        \begin{mdframed}[backgroundcolor=magenta!5]
          \begin{mdframed}[backgroundcolor=gray!20]
            Non-$\ell^1$-splicing (\Cref{thm:qual:2}).
          \end{mdframed}
          \hfill\newline Here, we will apply \Cref{lem:splicing} with
          \begin{align*}
            \mathcal{I}
            = \{ k\in\Lambda_{i-1}^1 : k > \max \mathcal{B}_{i-1}\},\\
            \eta = \eta_i/C_u,\quad
            y_j = d_j,\ 1\leq j\leq i-1
          \end{align*}
          to obtain an infinite set $\Lambda_i^0\subset \mathcal{I}$ such that
          \begin{equation}\label{eq:26}
            \sup_{\|x\|_{X}\leq 1} |\langle \widetilde{T} P_{\Lambda_i^0} x, d_j \rangle| \leq \eta_i/C_u
          \end{equation}
          for all $1\leq j\leq i-1$.
        \end{mdframed}
      \end{minipage}
    \end{mdframed}
    In both cases, using \Cref{lem:annihil} with $\mathcal{I} = \Lambda_i^0$, $\eta=\eta_i$,
    $x_j = T b_j$, $y_j = f_j$, $1\leq j\leq i-1$ yields an infinite set
    $\Lambda_i^1\subset\Lambda_i^0$ such that
    \begin{equation}\label{eq:27}
      \sup \Bigl\{
      \Big|\Big\langle
      \widetilde{T} b_j,
      \sum_{k\in \mathcal{B}} \varepsilon_k f_k\Big\rangle\Big| :
      \mathcal{B}\subset\Lambda_i^1,\ |\mathcal{B}| = 2,\
      \varepsilon\in \mathcal{E}(\mathcal{B}),\ 1\leq j\leq i-1\Bigr\}
      \leq \eta_i.
    \end{equation}
    Now we select any $\mathcal{A}_i\subset \Lambda_i^1$ with $|\mathcal{A}_i|=N$.  Applying
    \Cref{pro:subsymm} to $\mathcal{A} = \mathcal{A}_i$ gives us a set
    $\mathcal{B}_i\subset\mathcal{A}_i$ with $|\mathcal{B}_i|=2$ and
    $(\varepsilon_k)\in\mathcal{E}(\mathcal{B}_i)$ such that
    \begin{equation}\label{eq:28}
      \langle \widetilde{T} b_i, d_i\rangle
      \geq 2(1-\kappa)\delta,
    \end{equation}
    where we put
    \begin{equation}\label{eq:29}
      b_i
      = \sum_{k\in\mathcal{B}_i} \varepsilon_k e_k
      \qquad\text{and}\qquad
      d_i
      = \sum_{k\in\mathcal{B}_i} \varepsilon_k f_k.
    \end{equation}
    This concludes the inductive construction of our block basis.  To summarize, we proved the
    following estimates.  In both cases, we obtain from~\eqref{eq:27} and~\eqref{eq:28}
    \begin{equation}\label{eq:30}
      |\langle \widetilde{T} b_j, d_i \rangle|
      \leq \eta_i
      \quad\text{and}\quad
      \langle \widetilde{T} b_i, d_i\rangle
      \geq 2(1-\kappa)\delta,
      \qquad i,j\in\mathbb{N},\ j\leq i-1.
    \end{equation}
    Using different methods, we obtained the following estimates in each case.
    \begin{mdframed}[backgroundcolor=black!5,leftline=false,rightline=false,topline=false,bottomline=false]
      \begin{minipage}[t]{.5\linewidth}
        \begin{mdframed}[backgroundcolor=cyan!5]
          \begin{mdframed}[backgroundcolor=gray!20]
            Schauder basis (\Cref{thm:qual:1}).
          \end{mdframed}
          \hfill\newline Since $\mathcal{B}_i\subset\Lambda_i^1\subset\Lambda_i^0$ with
          $|\mathcal{B}_i|=2$, we obtain from~\eqref{eq:25}
          \begin{equation}\label{eq:31}
            |\langle \widetilde{T} b_i, d_j\rangle|
            \leq \eta_i
          \end{equation}
          for all $i,j\in\mathbb{N}$ with $j\leq i-1$.
        \end{mdframed}
      \end{minipage}
      \begin{minipage}[t]{.5\linewidth}
        \begin{mdframed}[backgroundcolor=magenta!5]
          \begin{mdframed}[backgroundcolor=gray!20]
            Non-$\ell^1$-splicing (\Cref{thm:qual:2}).
          \end{mdframed}
          \hfill\newline Since
          $\mathcal{B}_i\subset\Lambda_i^1\subset\Lambda_i^0\subset\Lambda_{i-1}^1$ with
          $|\mathcal{B}_i|=2$, we obtain from~\eqref{eq:26}
          \begin{equation}\label{eq:32}
            \Bigl|\Bigl\langle
            \widetilde{T} \sum_{j=i+1}^\infty a_j b_j,
            d_i
            \Bigr\rangle\Bigr|
            \leq \frac{\eta_i}{C_u}
            \Bigl\|\sum_{j=i+1}^\infty a_j b_j\Bigr\|
          \end{equation}
          for all $i,j\in\mathbb{N}$ with $1\leq j\leq i-1$ and all $(a_j)$ such that the tail
          series $\sum_{j=i+1}^\infty a_j b_j$ converges in the $\sigma(X,Y)$ topology.
        \end{mdframed}
      \end{minipage}
    \end{mdframed}
  \end{proofstep}
  
  \begin{proofstep}[Step~\theproofstep: Factorization]\label{proof:thm:qual:2:step:ii}
    This step of the proof is similar to Step~3 in the proof of
    \cite[Theorem~1.1]{lechner:2019:subsymm}.  The most notable differences are the following: The
    operators $B$ and $Q$ have been modified, the additional parameter $\delta$ is now involved, the
    estimate for $P\widetilde T z - z$ is different in the proof of \Cref{thm:qual:1}.  For those
    reasons, we will now give a concise version of this argument.

    Let $B,Q$ be defined as in~\eqref{eq:BQ} and put $Z = B(X)$.  Utilizing \Cref{pro:emb-proj}
    yields
    \begin{equation}\label{eq:33}
      2\cdot I_{X} = QB
      \qquad\text{and}\qquad
      \|B\|, \|Q\| \leq 2C_uC_s.
    \end{equation}
    Next, we define $P\colon X\to Z$ by
    \begin{equation}\label{eq:U}
      Px
      = \sum_{j=1}^\infty
      \frac{\langle x, d_j\rangle}
      {\langle \widetilde{T} b_j,
        d_j\rangle}
      b_j,
      \qquad x\in X,
    \end{equation}
    and observe that by the large diagonal of $\widetilde{T}$ (see~\eqref{eq:20}), the
    unconditionality of $(e_j)$ (see~\eqref{eq:unconditional}) and \Cref{pro:emb-proj} we obtain
    \begin{equation}\label{eq:34}
      \|Px\|_{X}
      \leq \frac{C_u^2 C_s}{(1-\kappa)\delta} \|x\|_{X},
      \qquad x\in X.
    \end{equation}

    For each of the two cases, let $z = \sum_{j=1}^\infty a_j b_j\in Z$.
    \begin{mdframed}[backgroundcolor=black!5,leftline=false,rightline=false,topline=false,bottomline=false]
      \begin{minipage}[t]{.5\linewidth}
        \begin{mdframed}[backgroundcolor=cyan!5]
          \begin{mdframed}[backgroundcolor=gray!20]
            Schauder basis (\Cref{thm:qual:1}).
          \end{mdframed}
          \hfill\newline Since $(e_j)$ is a Schauder basis, we can exploit the continuity of
          $\widetilde{T}$ and obtain that
          \begin{equation}\label{eq:35}
            \begin{split}
              P\widetilde{T}z - z &= \sum_{i=1}^\infty \sum_{j=1}^{i-1} a_j \frac{\langle
                \widetilde{T} b_j, d_i\rangle}{\langle \widetilde{T} b_i, d_i\rangle} b_i\\
              &+ \sum_{i=1}^\infty \sum_{j=i+1}^\infty a_j \frac{\langle \widetilde{T} b_j,
                d_i\rangle} {\langle \widetilde{T} b_i, d_i\rangle} b_i.
            \end{split}
          \end{equation}
        \end{mdframed}
      \end{minipage}
      \begin{minipage}[t]{.5\linewidth}
        \begin{mdframed}[backgroundcolor=magenta!5]
          \begin{mdframed}[backgroundcolor=gray!20]
            Non-$\ell^1$-splicing (\Cref{thm:qual:2}).
          \end{mdframed}
          \hfill\newline In general, $(e_j)$ is not a Schauder basis, thus we \emph{cannot} exploit
          the continuity of $\widetilde{T}$ and we are left with
          \begin{equation}\label{eq:36}
            \begin{split}
              P\widetilde{T}z - z &= \sum_{i=1}^\infty \sum_{j=1}^{i-1} a_j \frac{\langle
                \widetilde{T} b_j, d_i\rangle} {\langle \widetilde{T} b_i, d_i\rangle} b_i\\
              &+ \sum_{i=1}^\infty \frac{\bigl\langle \widetilde{T} \sum_{j=i+1}^\infty a_j b_j,
                d_i\bigr\rangle} {\langle \widetilde{T} b_i, d_i\rangle} b_i.
            \end{split}
          \end{equation}
        \end{mdframed}
      \end{minipage}
    \end{mdframed}
    Note that $2 |a_j| = |\langle z, b_j\rangle| \leq 2 C_d \|z\|_{X}$.  This estimate together
    with~\eqref{eq:25}, \eqref{eq:30} and~\eqref{eq:21} yields (apply the triangle inequality
    to~\eqref{eq:35})
    \begin{equation}\label{eq:37}
      \|P\widetilde{T}z - z\|_{X}
      \leq \frac{\kappa}{(1-\kappa)} \|z\|_{X},
      \qquad z\in Z.
    \end{equation}
    The same result can be obtained by applying the triangle inequality to~\eqref{eq:36} and
    using~\eqref{eq:26}, \eqref{eq:30} and~\eqref{eq:21}.

    Define $J\colon Z\to X$ by $Jz = z$ and note that by~\eqref{eq:37}
    $P\widetilde{T}J\colon Z\to Z$ is invertible.  Hence, if we define
    $V = (P\widetilde{T}J)^{-1}P$, then by~\eqref{eq:34} and~\eqref{eq:37} we obtain
    \begin{equation}\label{eq:38}
      I_Z = V\widetilde{T}J
      \qquad\text{and}\qquad
      \|V\|
      \leq \frac{C_u^2 C_s}{(1-2\kappa)\delta}.
    \end{equation}
    Combining~\eqref{eq:33} with~\eqref{eq:38} and recalling that $\widetilde{T} = TM$ yields
    \begin{equation}\label{eq:39}
      I_X
      = QI_ZB/2
      = QV\widetilde{T}JB/2
      = QVTMJB/2
      = ETF,
    \end{equation}
    where we put $E = QV$ and $F = MJB/2$.  By~\eqref{eq:33}, \eqref{eq:38}, \eqref{eq:21} and since
    $\|M\|\leq C_u$, we obtain
    \begin{equation*}
      \|E\|\cdot\|F\|
      \leq \frac{2C_u^5 C_s^3}{(1-2\kappa)\delta}
      \leq \frac{2C_u^5 C_s^3}{\delta} + \eta,
    \end{equation*}
    which together with~\eqref{eq:39} proves both theorems.\qedhere
  \end{proofstep}
\end{myproof}

\begin{rem}
  It is conceivable to prove \Cref{thm:qual:1} by adapting the gliding-hump techniques
  in~\cite{casazza:lin:1974,casazza:kottman:lin:1977}.  However, these techniques do not appear to
  work in the non-separable case: The exact place in the proof where the seemingly insurmountable
  problem occurs is illustrated by comparing~\eqref{eq:35} with~\eqref{eq:36}.  In order to get from
  the formula in~\eqref{eq:36} to that of~\eqref{eq:35}, we need the following identity to be true:
  \begin{equation}\label{eq:51}
    \bigl\langle \widetilde{T} \sum_{j=i+1}^\infty a_j b_j,
    d_i\bigr\rangle
    = \sum_{j=i+1}^\infty a_j \langle \widetilde{T} b_j, d_i\rangle.
  \end{equation}
  In the non-separable case $\sum_{j=i+1}^\infty a_j b_j$ does not need to converge in the
  norm-topology, therefore we cannot use the norm-continuity of the operator $\widetilde T$ to swap
  the operator with the sum.

  To compensate we work with the lefthand side of~\eqref{eq:51} directly by exploiting that $(e_j)$
  is non-$\ell^1$-splicing (see~\cite{lechner:2019:subsymm} for a prior use of this approach),
  whereby we preselect a large subspace whose image under $T$ annihilates all previously constructed
  block basis elements.
\end{rem}


\section{Direct sums of Banach spaces with a subsymmetric basis}
\label{sec:direct-sums-banach}

Here, we will provide a proof for \Cref{thm:qual:3}.  The relationship between the proof
of~\cite[Theorem~1.1]{lechner:2019:subsymm} and the proof
of~\cite[Theorem~1.2]{lechner:2019:subsymm} is very similar to the relationship between the proof of
\Cref{thm:qual:2} and that of \Cref{thm:qual:3}: when the array $(e_{n,j})_{n,j}$ is linearized
correctly, we can essentially use the same construction as for the one-parameter basis $(e_j)$.
Instead of repeating large portions of a lengthy proof, we will describe the necessary adaptations.
\begin{proof}[Proof of \Cref{thm:qual:3}]
  First, we need two-parameter versions of \Cref{lem:annihil}, \Cref{lem:splicing} and
  \Cref{pro:subsymm}.  For the former two, we refer to~\cite[Lemma~5.1,
  Lemma~5.2]{lechner:2019:subsymm}.  The two-parameter version of \Cref{pro:subsymm} is just the
  coordinate-wise application of \Cref{pro:subsymm} (i.e.\@ replacing $(e_j)_j$ and $(f_j)_j$ with
  $(e_{n,j})_j$ and $(f_{n,j})_j$ for fixed $n$).

  Heavily exploiting that $(e_j)$ is non-$\ell^1$-splicing, it was possible to construct the block
  basis $(b_j)$ in the proof of~\cite[Theorem~1.2]{lechner:2019:subsymm} with the same two basic
  steps as in the proof of~\cite[Theorem~1.1]{lechner:2019:subsymm}.  The $i^{\text{th}}$ step of
  the construction reads as follows:
  \begin{enumerate}[(P1)]
  \item\label{item:6} annihilating the previously constructed vectors $T b_j$, $1\leq j\leq i-1$ by
    choosing $b_i$ in $[e_k : k\in\mathcal{A}_i]$;
  \item\label{item:7} annihilating vectors that will be constructed in later steps by preselecting
    infinitely many suitable coordinates $\mathcal{A}_{i+1}$.
  \end{enumerate}
  
  For the proof of \Cref{thm:qual:3}, we can use the same scheme as described above, i.e.\@ we
  replace the two basic construction steps~\textrefp[P]{item:6} and~\textrefp[P]{item:7} by the
  following three steps described in the $i^{\text{th}}$ inductive step in the proof of
  \Cref{thm:qual:2}:
  \begin{enumerate}[(F1)]
  \item\label{item:8} using \Cref{lem:splicing} to preselect infinitely many coordinates
    $\Lambda_i^0$ so that vectors in $T P_{\Lambda_i^0}(X)$ annihilate the previously constructed
    $d_j$, $1\leq j\leq i-1$;
  \item\label{item:9} preselecting another infinite set $\Lambda_i^1\subset\Lambda_i^0$ by utilizing
    \Cref{lem:annihil}, so that we can choose among many disjointly supported candidates in
    $[e_k : k\in\Lambda_i^1]$ for the current block basis element $d_i$, all of which annihilate
    previously constructed vectors $\widetilde T b_j$, $1\leq j\leq i-1$;
  \item\label{item:10} using \Cref{pro:subsymm} to find among those candiates block basis elements
    $b_i$ and $d_i$ so that the diagonal is kept large, i.e.\@
    $\langle\widetilde T b_i, d_i \rangle\geq 2(1-\kappa)\delta$.
  \end{enumerate}
  Finally, we note that Step~3 in the proof of~\cite[Theorem~1.2]{lechner:2019:subsymm} reduces to
  the following case: we have $\mathcal{I} = \mathbb{N}$,
  $\mathcal{J}_i = \mathcal{K}_i = \mathbb{N}$ and $H=\widetilde T$.  Also, that reduced case is
  essentially \textref[Step~]{proof:thm:qual:2:step:ii} in the proof of \Cref{thm:qual:2}.
  Rerunning those arguments with the described modifications concludes the proof.
\end{proof}

We will now show how to obtain~\cite[Theorem~1.2]{lechner:2019:subsymm} from \Cref{thm:qual:3} (see
\Cref{cor:qual:3}, below).
\begin{cor}[{\cite[Theorem~1.2]{lechner:2019:subsymm}}]\label{cor:qual:3}
  Let the dual pair $(X,Y,\langle\cdot,\cdot\rangle)$ of infinite dimensional Banach spaces $X$ and
  $Y$ and the sequences $(e_j)$, $(f_j)$ satisfy \textrefp[B]{item:1}--\textrefp[B]{item:5} with
  constant $C_d$.  Assume that $(e_j)_j$ is subsymmetric and non-$\ell^1$-splicing.  Then the Banach
  spaces $\ell^p(X)$, $1\leq p\leq\infty$ are primary.
\end{cor}

\begin{proof}
  Let $Q\colon\ell^p(X)\to\ell^p(X)$ be a bounded projection, and let $(e_{n,j})_j$ and
  $(f_{n,j})_j$ denote a copies of $(e_j)_j$ and $(f_j)_j$ in the $n^{\text{th}}$ coordinate of the
  direct sum.  Note that at least one of the sets
  \begin{align*}
    &\bigl\{n\in\mathbb{N} : \text{the set}\
      \{j\in \mathbb{N} :
      |\langle Qe_{n,j}, f_{n,j}\rangle| \geq 1/2
      \}\ \text{is infinite}\bigr\},\\
    &\bigl\{n\in\mathbb{N} : \text{the set}\
      \{j\in \mathbb{N} :
      |\langle (I_{\ell^p(X)}-Q) e_{n,j}, f_{n,j}\rangle| \geq 1/2
      \}\ \text{is infinite}\bigr\},
  \end{align*}
  is infinite.  If the first set is infinite we put $H = Q$, and we define $H = I_{\ell^p(X)} - Q$
  otherwise.  In any case, there exist infinite sets $\mathcal{I}\subset\mathbb{N}$ and
  $\mathcal{J}\subset\mathbb{N}$, $n\in\mathcal{I}$ so that $H$ has large diagonal with respect to
  $(e_{n,j} : n\in\mathcal{I},\ j\in\mathcal{J}_n)$.  We define
  $Z_n = \overline{\spn\{ e_{n,j} : j\in\mathcal{J}_n\}}^{\sigma(X,Y)}$, $n\in\mathcal{I}$ and put
  $Z = \ell^p((Z_n : n\in\mathcal{I}))$.  Let $P\colon \ell^p(X)\to Z$ denote the natural projection
  onto $Z$, i.e.\@ $P(e_{n,j}) = e_{n,j}$ if $n\in\mathcal{I}$ and $j\in\mathcal{J}_n$ and
  $P(e_{n,j}) = 0$ otherwise.  We record that by the unconditionality of $(e_{n,j})$, the projection
  $P$ is bounded.  We note that by the subsymmetry of $(e_{n,j} : j\in\mathcal{J}_n)$, we have that
  $Z_n$ is isomorphic to $X$; hence, $Z$ is isomorphic to $\ell^p(X)$ and we denote that isomorphism
  by $R\colon \ell^p(X)\to Z$.  By \Cref{thm:qual:3},
  $(e_{n,j} : n\in\mathcal{I},\ j\in\mathcal{J}_n)$ has the factorization property; in particular,
  $I_{Z}$ factors through the operator $PH\colon Z\to Z$, i.e.\@ there exist bounded operators
  $A,B\colon Z\to Z$ such that $I_Z = BPHA$.  Consequently, $I_{\ell^p(X)} = R^{-1}BPHAR$, that is
  $I_{\ell^p(X)}$ factors through $H$.  Using Pe{\l}czy{\'n}ski's decomposition method (see e.g.\@
  Step~4 in the proof of~\cite[Theorem~1.2]{lechner:2019:subsymm}) concludes the proof.
\end{proof}


\section{Finite dimensional quantitative factorization results}
\label{sec:quant-fact-results}

We first prove \Cref{thm:fact:quant:uncond} and then conclude this work with the proof of
\Cref{cor:fact:quant:subs}.

In this section, $(e_j)$ denotes a normalized basis for the Banach space $X$ and $(e_j^*)$ denotes
the biorthogonal functionals to $(e_j)$.  Recall that in~\eqref{eq:fundamentals:3}, we defined the
function $\tau\colon\mathbb{N}\to[0,\infty)$ by
\begin{equation*}
  \tau(n)
  = \max\Bigl\{
  \min\Bigl(
  \max_{1\leq j\leq n} \Bigl\|\sum_{\substack{i=1\\i\neq j}}^n\varepsilon_{ij} e_i\Bigr\|_X,
  \max_{1\leq i\leq n} \Bigl\|\sum_{\substack{j=1\\j\neq i}}^n\varepsilon_{ij} e_j^*\Bigr\|_{X^*}
  \Bigr) : \varepsilon_{ij}\in\{\pm 1\},\ 1\leq i,j\leq n
  \Bigr\}.
\end{equation*}

Mainly, the proof of \Cref{thm:fact:quant:uncond} follows the method used by Bourgain and Tzafriri
for \cite[Theorem~6.1]{bourgain:tzafriri:1987}, but instead of exploiting the lower $r$-estimate
which allows them to extract a large submatrix with small entries, we use our upper estimate for
$\tau(n)$.
\begin{proof}[Proof of \Cref{thm:fact:quant:uncond}]
  We begin by defining the constants
  \begin{align}
    \label{eq:54}
    \kappa_1 &= \min\bigl(1, \eta\bigr)/4,
    &\kappa_2 &= \min\bigl(\eta, (1+\eta)^{-1}\bigr)/4,
    &\alpha &= \sqrt{\frac{\delta\kappa_1}{\Gamma}}\cdot\frac{1}{\sqrt{n\tau(n)}},
  \end{align}
  and note that $\alpha\leq 1$ by the lower estimate in~\eqref{eq:thm:fact:quant:1}.  Now let
  $(\xi_i)_{i=1}^n$ denote independent random variables over the probability space
  $(\Omega,\Sigma,\prob)$ taking values in $\{0,1\}$ with $\cond \xi_i = \alpha$, where $\cond$
  denotes the conditional expectation.  For each $\omega\in\Omega$ we define the operators
  \begin{equation}\label{eq:46}
    A(\omega)
    = R(\omega)D^{-1}TR(\omega)
    \qquad\text{and}\qquad
    R(\omega) e_i
    = \xi_i(\omega) e_i,\ 1\leq i\leq n,
  \end{equation}
  and note that
  \begin{equation}\label{eq:47}
    A\Bigl(\sum_{i=1}^n a_i e_i\Bigr)
    = \sum_{i,j=1}^n a_i\xi_i\xi_j \frac{\langle T e_i, e_j^*\rangle}{\langle T e_j, e_j^*\rangle} e_j.
  \end{equation}
  Let $x = \sum_{i=1}^n a_i e_i$ and observe that by\eqref{eq:47}
  \begin{equation*}
    \|(A - R)x\|
    \leq \sum_{i\neq j} |a_i|\xi_i\xi_j
    \frac{|\langle T e_i, e_j^*\rangle|}{|\langle T e_j, e_j^*\rangle|}
    \leq \frac{\|x\|}{\delta} \sum_{i\neq j} \xi_i\xi_j
    |\langle T e_i, e_j^*\rangle|.
  \end{equation*}
  Taking the supremum over all $\|x\|\leq 1$ and taking the expectation yields
  \begin{equation}\label{eq:48}
    \cond\|A-R\|
    \leq \frac{1}{\delta} \sum_{i\neq j} \cond \xi_i\xi_j
    |\langle T e_i, e_j^*\rangle|
    = \frac{\alpha^2}{\delta} \sum_{i\neq j} |\langle T e_i, e_j^*\rangle|.
  \end{equation}
  Defining $\varepsilon_{ij} = \sign(\langle T e_i, e_j^*\rangle)$ we obtain from~\eqref{eq:48}
  \begin{equation}\label{eq:49}
    \cond\|A-R\|
    \leq \frac{\alpha^2}{\delta}
    \sum_{i} \bigl\langle T e_i, \sum_{j:j\neq i} \varepsilon_{ij} e_j^*\bigr\rangle
    \leq \alpha^2\frac{\Gamma}{\delta} n \max_i \bigl\|\sum_{j:j\neq i} \varepsilon_{ij}e_j^*\bigr\|.
  \end{equation}
  Similarly, we obtain
  \begin{equation}\label{eq:50}
    \cond\|A-R\|
    \leq \alpha^2\frac{\Gamma}{\delta} n \max_j \bigl\|\sum_{i:i\neq j} \varepsilon_{ij}e_i\bigr\|.
  \end{equation}
  Combing~\eqref{eq:49}, \eqref{eq:50}, \eqref{eq:fundamentals:3} and~\eqref{eq:54} yields
  \begin{equation}\label{eq:52}
    \cond\|A-R\|
    \leq\alpha^2\frac{\Gamma}{\delta} n\tau(n)
    = \kappa_1.
  \end{equation}

  Now define
  \begin{equation*}
    \Omega'
    = \{\omega\in\Omega : \Bigl|\sum_{i=1}^n \xi_i(\omega) - \alpha n\Bigr| \leq \alpha n/2\}
  \end{equation*}
  and observe that by~\eqref{eq:54} and~\eqref{eq:thm:fact:quant:1}
  \begin{equation*}
    \prob(\Omega'^c)
    = \prob\Bigl(\Bigl\{\omega\in\Omega : \Bigl|\sum_{i=1}^n \xi_i - \alpha n\Bigr| > \alpha n/2\Bigr\}\Bigr)
    \leq\frac{4}{\alpha^2n^2}\cond\biggl( \sum_{i=1}^n \xi_i - \alpha n\biggr)^2
    = \frac{4(1-\alpha)}{\alpha n}
    \leq \kappa_2.
  \end{equation*}
  Thus, combining this measure estimate with~\eqref{eq:52}, we can find an $\omega_0\in \Omega'$
  such that
  \begin{equation}\label{eq:53}
    \|A(\omega_0)-R(\omega_0)\|
    \leq \kappa_1/(1-\kappa_2).
  \end{equation}
  Since $\kappa_1/(1-\kappa_2) < 1$, $A(\omega_0)$ is invertible, hence, we obtain by~\eqref{eq:54}
  \begin{equation}\label{eq:61}
    \|A^{-1}(\omega_0)\|
    \leq \frac{1-\kappa_2}{1-\kappa_1-\kappa_2}
    \leq 1 + \eta.
  \end{equation}
  We put $\sigma = \{1\leq i\leq n : \xi_i(\omega_0) = 1\}$ note that since $\omega_0\in\Omega'$ we
  have $|\sigma|\geq \alpha n/2$; appealing to~\eqref{eq:54} shows~\eqref{eq:thm:fact:quant:3}.
  Since $R(\omega_0) = R_\sigma$, \eqref{eq:thm:fact:quant:4} follows from~\eqref{eq:61} and the
  definition of $A$ (see~\eqref{eq:46}).

  We conclude the proof by defining $E\colon X_\sigma\to X_n$ by $Ex = x = R_\sigma x$ and
  $P\colon X_n\to X_\sigma$ by $P = (R_\sigma D^{-1}TR_\sigma)^{-1}R_\sigma D^{-1}$ and using that
  $\|D^{-1}\|\leq C_u/\delta$ together with~\eqref{eq:thm:fact:quant:4}.
\end{proof}

To conclude, we show that \Cref{thm:fact:quant:uncond} implies
\cite[Theorem~1.2]{lechner:2019:subsymm}, below.
\begin{proof}[Proof of \Cref{cor:fact:quant:subs}]
  Recalling~\eqref{eq:fundamentals:1}, we observe that if $(e_j)$ is $C_u$-unconditional, then
  \begin{equation}
    \label{eq:57}
    \tau(n)
    \leq C_u\nu(n),
    \qquad n\in\mathbb{N}.
  \end{equation}
  Moreover, if $(e_j)$ is $(C_u,C_s)$-subsymmetric, then by~\eqref{eq:57} and~\eqref{eq:58} we
  obtain
  \begin{equation}
    \label{eq:59}
    \tau(n)
    \leq C_uC_s\min(\lambda(n-1),\mu(n-1)),
    \qquad n\in\mathbb{N}.
  \end{equation}
  Thus, by~\eqref{eq:14}, we also have
  \begin{equation}
    \label{eq:60}
    \tau(n)
    \leq \sqrt{2C_u^3C_s^3(n-1)},
    \qquad n\in\mathbb{N}.
  \end{equation}

  Using~\eqref{eq:60}, \eqref{eq:62} and that $\tau(n)\geq 1/C_u$, $n\geq 2$, we
  obtain~\eqref{eq:thm:fact:quant:1}; thus, \Cref{thm:fact:quant:uncond} yields \eqref{eq:63}
  and~\eqref{eq:64}.  Noting that $X_k$ is $C_s$-isomorphic to $X_\sigma$, we obtain the estimate
  $\|E\|\|P\|\leq C_u^2C_s^2\frac{1+\eta}{\delta}$.
\end{proof}


\bibliographystyle{abbrv}
\bibliography{bibliography}

\end{document}